\documentclass[11pts]{amsart}
\usepackage{amssymb}
\usepackage{epsfig}
\usepackage{amscd}
\usepackage{graphicx}
\newtheorem{theorem}{Theorem}[section]
\newtheorem{prop}[theorem]{Proposition}
\newtheorem{cor}[theorem]{Corollary}
\newtheorem{lem}[theorem]{Lemma}
\newtheorem{rem}[theorem]{Remark}
\newtheorem{defn}[theorem]{Definition}
\newtheorem{example}[theorem]{Example}
\newtheorem{rems}[theorem]{Remarks}
\def\ra{\rightarrow}

\def\b{\beta}

\def\a{\alpha}
\def\mc{\mathcal}
\def\lan{\langle}
\def\ran{\rangle}
\def\un{\underline}
\def\d{\Delta}
\begin{document}
\title{Representation stability of power sets and square free polynomials}
\thanks{2010 AMS Classification Primary: 20C30, 13A50; Secondary: 20F36, 55R80.\\
\indent This research is partially supported by Higher Education Commission, Pakistan.}
\author{Samia Ashraf$^{1}$, Haniya Azam$^{2}$, Barbu Berceanu$^{3}$}
\address{$^{1}$Abdus Salam School of Mathematical Sciences,
 GC University, Lahore-Pakistan.}
\email {samia.ashraf@yahoo.com}
\address{$^{2}$Abdus Salam School of Mathematical Sciences,
 GC University, Lahore-Pakistan.}
\email {centipedes.united@gmail.com}
\address{$^{3}$Abdus Salam School of Mathematical Sciences,
 GC University, Lahore-Pakistan, and
 Institute of Mathematics Simion Stoilow, Bucharest-Romania.(Permanent address)}
\email {Barbu.Berceanu@imar.ro}

\maketitle

\maketitle
\pagestyle{myheadings} \markboth{\centerline {\scriptsize
SAMIA ASHRAF,\,\,\,HANIYA AZAM ,\,\,BARBU BERCEANU}}{\centerline {\scriptsize
Representation stability of power sets and square free polynomials}}
\maketitle
\begin{abstract} The symmetric group acts on the power set and also on the set of square free polynomials. These two related representations
are analyzed from the stability point of view. An application is given for the action of the symmetric group on the cohomology of the pure braid group.
\end{abstract}

\textwidth=13cm

%\subjectclass{Ams}
\pagestyle{myheadings}
\section{Introduction}

The symmetric group $\mc{S}_n$ acts naturally on the power set $\mc{P}(n)$ of the set $\underline{n}=\{1,2,\ldots,n\}$:
$\pi\cdot A=\pi(A)$. It is obvious that the orbits of this action are
$\mc{P}_k(n)=\{A\subset\underline{n}\mid \mathrm{card}(A)= k\}$  for $k=0,1,\ldots,n$. More interesting is the linear
representation of the symmetric group on the linear space $L\mc{P}(n)$, the $\mathbb{Q}$-span of the power set: the
${\mc S}_n$-submodules $L\mc{P}_k (n)$ are not irreducible. We decompose them into irreducible $\mc{S}_n$-modules and we describe their bases, using the
equivalent representation of $\mc{S}_n$ onto the quotient ring of square free polynomials in $n$ variables
$\mc{S}f(n)=\mathbb{Q}[x_1,x_2,\ldots,x_n]\diagup \lan x_1^2,x_2^2,\ldots,x_n^2\ran$.
Next we analyze the sequences of these representations, $(\mc{P}(n))_{n\geq0}$ and $(\mc{S}f(n))_{n\geq0}$, and some related sequences from the stability point of view introduced by Church and Farb \cite{CF} for the representation ring $R(\mc{S}_n)$. We define an analogue of this stability for the Burnside ring $\Omega(\mc{S}_n)$ and we analyze the stability of the action of $\mc{S}_n$ on $\mc{P}(n)$, see Section 4.

For the irreducible ${\mc S}_n$-modules we will use the standard notation: $V_{\lambda}$ corresponds to the partition
$\lambda=(\lambda_1\geq \lambda_2\geq\ldots\geq\lambda_t\geq0)$ of $n$, and the stable notations of Church and
Farb (\cite{CF,C}) $V(\mu)_n=V_{(n-\sum \mu_i,\mu_1,\mu_2,\ldots,\mu_s)}$ for $\mu=(\mu_1\geq\mu_2\geq\ldots\geq\mu_s\geq0)$
satisfying the relation $n-\sum \mu_i\geq \mu_1$. Similarly, $U_\lambda$ is the Specht module and $U(\mu)_n$ is the Specht module
$U_{(n-\sum \mu_i,\mu_1,\mu_2,\ldots,\mu_s)}$. See \cite{FH}, \cite{J}, \cite{K} for references for the representation theory of $\mc{S}_n$.

Following \cite{CF} and \cite{C}, we say that a direct sequence
$X_*=(X_0\stackrel{\varphi_{0}}{\to}X_1 \stackrel{\varphi_{1}}{\to}\ldots X_n\stackrel{\varphi_{n}}{\to} X_{n+1}\ra\ldots)$
where $X_n$ is an $\mc{S}_n$-module, is \emph{consistent} if $\varphi_{n}$ is $\mc{S}_n$-equivariant. The sequence is \emph{injective} if $\varphi_n$ is eventually injective
and $\mc{S}_*$-\emph{surjective} if for $n$ large $\mc{S}_{n+1}. \mathrm{Im}(\varphi_{n})=X_{n+1}$. The sequence $X_*$ is \emph{representation stable}
if it satisfies the above conditions and also, for any stable type $\mu=(\mu_1,\mu_2,\ldots,\mu_s)$ of $\mc{S}_n$ modules, the sequence $(c_{\mu,n})_n$ of multiplicities of $V(\mu)_n$ in $X_n$ is eventually constant. The sequence is \emph{uniformly representation stable} if there is a natural number $N$, independent of $\mu$, such that for any $\mu$ and any $n\geq N$, $c_{\mu,n}=c_{\mu,N}$. The sequence $X_*$ is \emph{strongly representation stable} if for any type $\mu=(\mu_1,\mu_2,\ldots,\mu_s)$ the
$\mc{S}_{n+1}$-span of the image of the isotypical component of $V(\mu)_n$ in $X_n$ coincides with the isotypical component $V(\mu)_{n+1}$ in $X_{n+1}$. We say that a consistent sequence is \emph{monotonic} if for each $\mc{S}_n$ submodule $U\cong c\cdot V(\mu)_n$ in $X_n$, the $\mc{S}_{n+1}$-span of the image of $U$ in $X_{n+1}$ contains $ c\cdot V(\mu)_{n+1}$ as a submodule. See \cite{CF} and \cite{C} for other versions of representation stability.

The main result of Sections 2 and 3 is the next theorem.
%\begin{theorem}
\\{\bf Theorem A.}
\emph{The sequence of $\mc{S}_*$-modules $(L\mc{P}_k(n))_{n\geq0}$ with $$L\mc{P}_k(n)=V(0)_n\oplus V(1)_n\oplus \ldots\oplus V(k)_n$$ (for $n\geq 2k$) is consistent,
uniformly representation stable and monotonic.}
%\end{theorem}

In section 4 we define action stability for a sequence of $\mc{S}_*$-spaces $(X_n)_{n\geq0}$ and will show that the sequences $(\mc{P}_k(n))_{n\geq0}$ are strongly action stable and the sequence $(\mc{P}(n))_{n\geq0}$ is action stable: see Definition \ref{def4.4} and Proposition \ref{prop4.7}.

In the next section we transfer the results from $L\mc{P}_k(n)$ and $L\mc{P}(n)$ into the corresponding results for $\mc{S}f_k(n)$ and $\mc{S}f(n)$, the algebra of square free monomials. The Vi\`{e}te polynomials $\sigma_k^n=\mathop{\sum}\limits _{1\leq i_1\leq\ldots\leq i_k\leq n} x_{i_1}x_{i_2}\ldots x_{i_k}$ give a basis for the invariant part $\mc{S}f(n)^{\mc{S}_n}$. Our Proposition \ref{prop 5.8} is a generalization of this classical result: we describe canonical bases for all the irreducible $\mc{S}_n$-submodules of the square free polynomial algebra.

In Section 6 we apply some of the previous results to find the irreducible $\mc{S}_n$-modules of the quadratic part of the Arnold algebra, the cohomology algebra of the ordered configuration space of $n$ points in the plane. The stable cases, $n\geq 4$ for the first decomposition and $n\geq 7$ for the second, are given by:
%\begin{theorem}
\\{\bf Theorem B.}
\emph{The degree 1 and 2 components of the Arnold algebra decompose as } $$\mc{A}^1(n)=V(0)_n\oplus V(1)_n\oplus V(2)_n,$$ $$\mc{A}^2(n)=2V(1)_n\oplus 2V(2)_n\oplus2 V(1,1)_n\oplus V(3)_n \oplus 2V(2,1)_n \oplus V(3,1)_n.$$ These decompositions are given in \cite{CF} without proofs.
%\end{theorem}

\section{Canonical ${\mc S}_n$ filtration on $L\mathcal{P}_k(n)$ }

First we can assume $k\leq \lfloor \frac{n}{2}\rfloor$ because of the next obvious proposition:
\begin{prop}\label{prop 2.1}
a) For any $k$, $0\leq k\leq n $ the complementary map $C$ is ${\mc S}_n$-equivariant:
$$ C: \mc{P}_k(n)\rightarrow \mc{P}_{n-k}(n).$$

b) The ${\mc S}_n$ representations $L\mc{P}_{k}(n)$ and $L\mc{P}_{n-k}(n)$ are isomorphic.
\end{prop}

\noindent Now we will define, for $0\leq k\leq \frac{n}{2}$, a \emph{canonical filtration}
$$F_\ast L\mc{P}_k(n):\,\,0<F_0 L\mc{P}_k(n)<F_1 L\mc{P}_k(n)<\ldots <F_k L\mc{P}_k(n)=L\mc{P}_k(n)$$
with ${\mc S}_n$-submodules as follows: for $A\in \mc{P}_i(n),\,0\leq i\leq k$, denote by
$\sigma_k^n(A)$ the element of $ L\mc{P}_k(n)$ given by
$$ \sigma_k^n(A)=\mathop{\sum}\limits_{B\in \mc{P}_{k-i}(\un{n}\setminus A)}A\sqcup B $$
and the ${\mc S}_n$-submodule $F_i L\mc{P}_k(n)$ as the span $\mathbb{Q}\langle \sigma_k^n(A)\mid \mathrm{card}(A)=i\rangle$,
for $0\leq i\leq k $. For instance:
$$\begin{array}{lll}
F_0 L\mc{P}_2(4)  & = & \mathbb{Q}\langle \{1,2\}+\{1,3\}+\{1,4\}+\{2,3\}+\{2,4\}+\{3,4\}\rangle, \\
F_1 L\mc{P}_2(4)  & = & \mathbb{Q}\langle\{1,2\}+\{1,3\}+\{1,4\},\{1,2\}+\{2,3\}+\{2,4\}, \\
                  &   & \{1,3\}+\{2,3\}+\{3,4\},\{1,4\}+\{2,4\}+\{3,4\}\rangle,\\
F_2 L\mc{P}_2(4)  & = & \mathbb{Q}\langle\{1,2\},\{1,3\},\{1,4\},\{2,3\},\{2,4\},\{3,4\}\rangle.
                  \end{array}$$
\begin{lem}{\label{lemma 2.2}}
The family $ \{\sigma_k^n(A)\mid \mathrm{card}(A)=i\}$ is a basis of $F_iL\mc{P}_k(n)$.
\end{lem}
\begin{proof}(Here the condition $k\leq \frac{n}{2}$ is necessary). For $i=k$, $F_k L\mc{P}_k(n)=L\mc{P}_k(n)$ and
the claim is obvious. The restriction map
$$res: L\mc{P}_k(n)\rightarrow L\mc{P}_k(n-1), \,res(A)=\Bigg\{\begin{array}{cc}
                                                                   A & if\, A\subset \un{n-1} \\
                                                                   0 & if\, n\in A
                                                                 \end{array}$$
induces the sequence of vector spaces:
$$0\ra F_{i-1}L\mc{P}_{k-1}(n-1)\stackrel{\sqcup\{n\}}{\longrightarrow} F_i L\mc{P}_k(n)\stackrel{res}{\longrightarrow} F_i L\mc{P}_k(n-1)\ra 0. $$
The first arrow is injective, the second arrow is surjective, and the sequence is semiexact.  The first vector
space has dimension ${n-1}\choose {i-1}$ and by induction on $n$ the last vector space has dimension
${{n-1}\choose i}$ (even in the case $\frac{n-1}{2}<k\leq \lfloor \frac{n}{2}\rfloor$). The sequence gives
$\dim F_iL\mc{P}_k(n)\geq {{n-1}\choose {i-1}}+{{n-1}\choose i}={n \choose i}$. By definition of the space
$F_iL\mc{P}_k(n)$ its dimension is less or equal to $\mathrm{card}\{\sigma_k^n(A)\mid \mathrm{card}(A)=i\}={n\choose i}$, hence the result (and also the exactness of the sequence).
\end{proof}
\begin{lem}
The sequence $\{F_iL\mc{P}_k(n)\}_{0\leq i\leq k}$ is an increasing filtration of $\mc{S}_n$-modules.
\end{lem}
\begin{proof} The inclusion $F_i\subset F_{i+1}$ is a consequence of the equality
$$(k-i)\sigma_k^n(A)=\mathop{\sum}\limits_{p\notin A }\sigma_k^n(A\sqcup \{p\}).$$
The group $\mc{S}_n$ permutes the elements of the basis of $F_i$: $\pi \cdot \sigma_k^n(A)=\sigma_k^n(\pi(A))$.
\end{proof}
\begin{lem}
a) The $\mc{S}_n$-module $F_0L\mc{P}_k(n)$ is trivial.

b) The $\mc{S}_n$ representations $F_iL\mc{P}_k(n)$ and $L\mc{P}_i(n)$ are equivalent.
\end{lem}
\begin{proof} a) The space $F_0L\mc{P}_k(n)$ is generated by the invariant element
$$\sigma_k^n= \sigma_k^n(\emptyset)=\mathop{\sum}\limits_{A\in \mc{P}_k(n)}A. $$

b) Using Lemma \ref{lemma 2.2} the maps
$$ \varphi :F_iL\mc{P}_k(n)\leftrightarrows L\mc{P}_i(n):\psi$$
given by $\varphi (\sigma_k^n(A))=A$ and $\psi(A)=\sigma_k^n(A)$ are well defined, $\mc{S}_n$-equivariant and inverse to each other.
\end{proof}
The $\mc{S}_n$-module $L\mc{P}_k(n)$ is a classical object, a Specht module, and we will give a new proof for its decomposition into irreducible pieces.
\begin{prop}\label{prop1}
The $\mc{S}_n$-module $L\mc{P}_k(n)$ is isomorphic with the Specht module $U_{(n-k,k)}$.
\end{prop}
\begin{proof}
At the level of sets we have the equivariant bijective map:
$$\mc{P}_k(n)\ra \{\mbox{tabloids of type }(n-k,k)\}$$
given by
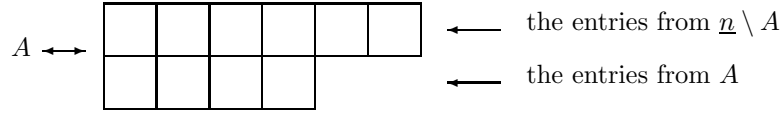
\begin{center}
\begin{picture}(200,50)
\put(-35,20){$A$}                        \put(-20,22){\vector(1,0){13}}   \put(-10,22){\vector(-1,0){13}}
\put(0,20){\line(1,0){120}}              \put(0,40){\line(1,0){120}}      \put(0,0){\line(1,0){80}}
\multiput(0,0)(20,0){5}{\line(0,1){40}}  \multiput(100,20)(20,0){2}{\line(0,1){20}}
\multiput(150,30)(0,-20){2}{\vector(-1,0){20}}
\put(160,30){$\mbox{the entries from }\underline{n}\setminus A $}
\put(160,10){$\mbox{the entries from }A$}
\end{picture}
\end{center}
\end{proof}
\noindent To describe the structure of the $\mc{S}_n$-modules $L\mc{P}_k(n)$ and $F_iL\mc{P}_k(n)$ we will use only
the fact that the Specht module $U_{(n-k,k)}$ contains $V_{(n-k,k)}$ (with some multiplicity).
\begin{prop}\label{prop 2.6}
The irreducible decompositions of the $\mc{S}_n$-modules $L\mc{P}_k(n)$ and $F_iL\mc{P}_k(n)$
$(0\leq i\leq k\leq \frac{n}{2})$ are given by:
$$\begin{array}{rll}
L\mc{P}_k(n)     & = &  V_{(n)}\oplus V_{(n-1,1)}\oplus \ldots \oplus V_{(n-k,k)};\\
F_iL\mc{P}_k(n)  & = & V_{(n)}\oplus V_{(n-1,1)}\oplus \ldots \oplus V_{(n-i,i)}.
\end{array}$$
\end{prop}
\begin{proof}
The proof is by induction on $n$ and $k$: using the imbedding $V_{(n-k,k)}< U_{(n-k,k)}\cong L\mc{P}_k(n)$ and
the Proposition \ref{prop1}, we have
$$ V_{(n)}\oplus V_{(n-1,1)}\oplus \ldots \oplus V_{(n-k+1,k-1)}\cong L\mc{P}_{k-1}(n)\cong F_{k-1}L\mc{P}_k(n)<L\mc{P}_k(n).$$
Using hook formula \cite{FH} we have $\dim V_{(n-k,k)}= {{n-1}\choose {k-1}} \frac{n-2k+1}{k}$ and counting the dimensions we find
$$\dim F_{k-1}L\mc{P}_k(n)+\dim V_{(n-k,k)}= {n\choose {k-1}} +{n\choose {k-1}}\frac{n-2k+1}{k}={n\choose {k}}$$
and this gives the direct sum $$L\mc{P}_{k-1}(n)\oplus V_{(n-k,k)}\cong F_{k-1}L\mc{P}_{k-1}(n)\oplus V_{(n-k,k)}\cong L\mc{P}_k(n).$$
\end{proof}
\begin{cor}
The $\mc{S}_n$-decomposition of the module $L\mc{P}(n)$ is given by:
$$L\mc{P}(n)=(n+1)V_{(n)}\oplus (n-1)V_{(n-1,1)}\oplus\ldots\oplus (n-2k+1)V_{(n-k,k)}\oplus\ldots\oplus
r V_{(\lceil\frac{n}{2}\rceil,\lfloor\frac{n}{2}\rfloor)},$$
where $ r=\lceil\frac{n}{2}\rceil-\lfloor\frac{n}{2}\rfloor+1$.
\end{cor}

A \emph{natural operation} on the power set $\mc{P}(n)$ satisfies $\pi(A\ast B)=\pi(A)\ast \pi(B)$ for any
permutation $\pi\in \mc{S}_n$. Given a natural operation $\psi$ (like $\cup ,\cap, \Delta,\ldots)$ on $\mc{P}(n)$,
we can linearize the map $\psi:\mc{P}(n)\times\mc{P}(n)\ra \mc{P}(n)$ and obtain an $\mc{S}_n$-map
$L\psi:L\mc{P}(n)\otimes L\mc{P}(n)\ra L\mc{P}(n)$. Irreducible decomposition of the tensor product
$L\mc{P}(n)^{\otimes2}$ will add more irreducible representations of $\mc{S}_n$:
$V_{(n-2,1,1)},\,V_{(n-3,2,1)},\,V_{(n-3,1,1,1)},\ldots$; all of them are contained in the kernel of $L\psi$.

\section{Representation Stability}\label{sec3}

Written using the stable notation, Proposition 2.6 gives the stable decompositions
$$\begin{array}{rlll}
L\mc{P}_k(n)    & = & V(0)_n\oplus V(1)_n\oplus\ldots\oplus V(k)_n  & \,\,\,\,(\mbox{for}\, n\geq2k),\\
F_iL\mc{P}_k(n) & = & V(0)_n\oplus V(1)_n\oplus\ldots\oplus V(i)_n  & \,\,\,\,(\mbox{for}\, n\geq2k\geq 2i).
\end{array}$$
We will show that the sequences $(L\mc{P}_k(n))_{n\geq 0}$, $(F_iL\mc{P}_k(n))_{n\geq 0}$ are consistent, uniformly representation
stable, and monotonic in the sense of [Ch] and [ChF]. The natural maps $\mc{P}(n)\stackrel{\varphi_n}{\to} \mc{P}(n+1)$ and
$\mc{P}_k(n)\stackrel{\varphi_{k,n}}{\to} \mc{P}_k (n+1)$ (and their linearizations
$L\mc{P}(n)\stackrel{L\varphi_n}{\to} L\mc{P}(n+1)$ and $L\mc{P}_k(n)\stackrel{L\varphi_{k,n}}{\to} L\mc{P}_k (n+1)$)
are induced by the inclusion map $\underline{n}\hookrightarrow\underline{n+1}$.

First we prove some "polynomial" identities:
\begin{lem}\label{lemma 3.1}
a) For an element $A$ in $\mc{P}_k(n)$ we have
$$ \sigma_{k}^{n}(A)=\sigma_{k}^{n+1}(A)-\sigma_{k}^{n+1}(A\sqcup \{n+1\}).$$

b) For $0\leq i\leq k-1$, $k\leq n$, we have
$$ \sigma^{n+1}_k(\un{i+1})=\frac{1}{(n-i)!(n-k+1)}\mathop{\sum}\limits_{\pi \in \mc{S}_{n+1}^{\geq i+1}}
\pi\cdot \sigma^n_k(\un{i})-(i+1,n+1)\cdot \sigma^n_k(\un{i}),$$
\noindent where $\mc{S}_{n+1}^{\geq i+1}$ is the subgroup of permutations fixing the elements $1,2,\ldots,i.$
\end{lem}
\begin{proof} a) The first equality is obvious: a term $A\sqcup B$, $B\subset \un{n}$ is contained  in
$ \sigma^n_k(A)$ and $\sigma^{n+1}_k(A)$ but not in the last sum, and a term $A\sqcup C \sqcup \{n+1\}$
is contained in the last two sums but not in the first one.

b) The sum $\mathop{\sum}\limits_{\pi \in \mc{S}_{n+1}^{\geq i+1}}\pi\cdot \sigma^n_k(\un{i})$ is symmetric in the
elements $i+1,i+2,\ldots,n,n+1$, and therefore all its terms have the same multiplicity. Multiplicity of the term
$\un{k}= \un{i} \sqcup \{i+1,\dots,k\}$, $k\leq n$, equals the number of permutations $\pi$ in  $ \mc{S}_{n+1}^{\geq i+1}$
sending a $k-i$ subset of $\{i+1,i+2,\ldots,n\}$ into $\{i+1,\ldots,k\}$ (because any sum $\pi \cdot \sigma_k^n(\un{i})$
contains $\un{k}$ at most once), and this number is given by
$${{n-i}\choose{k-i}}(k-i)!(n+1-k)!=(n-i)!(n+1-k).$$
Now the term $\un{k} =\un{i} \sqcup \{i+1,\ldots,k\}=\un{i+1}\sqcup \{i+2,\ldots,k\} $ appears in the average of the sum
$\mathop{\sum}\limits_{\pi \in \mc{S}_{n+1}^{\geq i+1}}\pi\cdot \sigma^n_k(\un{i})$ with coefficient $1$, like in
$\sigma^{n+1}_k(\un{i+1})$, and does not appear in the sum $\sigma^n_k(\un{i})$ modified by the transposition
$(i+1,n+1)$. The term $\un{i} \sqcup \{i+2,i+3,\ldots,k,n+1\}$ appears in the average of the sum
$\mathop{\sum}\limits_{\pi \in \mc{S}_{n+1}^{\geq i+1}}\pi\cdot \sigma^n_k(\un{i})$ with the same coefficient, $1$, and has
also the  coefficient $1$ in $(i+1,n+1)\cdot \sigma^n_k(\un{i})$.
\end{proof}
\begin{lem}\label{lemma 3.2}
a) For $0\leq i\leq k-1$, $k\leq n$ we have
$$\begin{array}{rll}
L \varphi_{k,n}(F_iL\mathcal{P}_k(n))              &  <  &  F_{i+1}L\mathcal{P}_k(n+1),\\
\mc{S}_{n+1}\cdot L \varphi_{k,n}(F_iL\mathcal{P}_k(n)) &  =  &  F_{i+1}L\mathcal{P}_k(n+1).
\end{array}$$

b) For $i=k\leq n$ we have
$$\begin{array}{rll}
L \varphi_{k,n}(F_kL\mathcal{P}_k(n))               & <  &  F_{k}L\mathcal{P}_k(n+1), \\
\mc{S}_{n+1}\cdot L \varphi_{k,n}(F_kL\mathcal{P}_k(n))  &  = &  F_{k}L\mathcal{P}_k(n+1).
\end{array}$$
\end{lem}
\begin{proof}
The part b) is obvious at the set level: $\varphi_{k,n}(\mathcal{P}_k(n))$ is part of $\mathcal{P}_k(n+1)$, and
$\mc{S}_{n+1}$ acts transitively on $\mathcal{P}_k(n+1)$. For part a), the first inclusion is a consequence of Lemma \ref{lemma 3.1} a) and
this inclusion implies
$$ \mc{S}_{n+1}\cdot L \varphi_{k,n}(F_iL\mathcal{P}_k(n))<F_{i+1}L\mathcal{P}_k(n+1).$$
For the reverse inclusion it is enough to show that $\sigma^{n+1}_k(\un{i+1})$ belongs to the $\mc{S}_{n+1}$ span of
the image of $F_iL\mathcal{P}_k(n)$, and this is a consequence of Lemma \ref{lemma 3.1} b).
 \end{proof}
\begin{rem}
It is also clear that the image of $F_iL\mathcal{P}_k(n)$ is not contained in $F_{i}L\mathcal{P}_k(n+1)$ for $i\leq k-1$.
\end{rem}

\begin{prop}
a) The sequence $(L\mc{P}(n),L\varphi _n)_{n\geq 0}$ is consistent, representation
stable, and monotonic.

b) The sequences $(L\mc{P}_k(n),L\varphi _{k,n})_{n\geq 0}$ are consistent, uniformly representation
stable, and monotonic.
\end{prop}
\begin{proof}
It is obvious that the maps $L\varphi_{k,n}: L\mc{P}_k(n)\ra Res_{\mc{S}_n}^{\mc{S}_{n+1}}L\mc{P}_k(n+1)$ are injective,
$\mc{S}_n $-equivariant and also that $\mc{S}_{n+1}\cdot {\rm Im}(\varphi_{k,n})=\mc{P}_k(n+1)$. The sequence of multiplicities of
$V(\mu)_n$ in $L\mc{P}_k(n)$ is constant 1 for $\mu =(i)$, $0\leq i\leq n/2$, and 0 for the other irreducible modules.
For monotonicity, we will use the previous lemma.\\
If $U$ is an isotypical $\mc{S}_n$-subspace of $L\mc{P}_k (n)$, $U\cong \alpha _\lambda V(\lambda)_{n}$, then $\lambda=(i)$ with
multiplicity $1$ (for some $i\in\{0,1,\ldots,k\}$); hence $U\nless F_{i-1}L\mc{P}_{k}(n)$ and $U< F_{i}L\mc{P}_{k}(n)$.
Lemma \ref{lemma 3.2} shows that $\mc{S}_{n+1}\cdot L\varphi _{k,n}(U)< F_{i+1}L\mc{P}_k(n+1)$ (or $F_kL\mc{P}_k(n+1)$ for $i=k$), which
contains $V(i)_{n+1}$ with multiplicity 1. On the other hand, $L\varphi_{k,n}(U)$ is contained in $F_{i-1}L\mc{P}_k(n+1)$,
therefore its span $\mc{S}_{n+1}\cdot L\varphi_{k,n}(U)$ should contain $V(i)_{n+1}$.

We have similar results for the whole power set, with the stable multiplicity 2, but not uniformly stable.
\end{proof}
\section{Stability of the symmetric group actions}{\label{sec4}}
We give a set-theoretical analogue of the representation stability for a (direct) sequence of finite $\mc{S}_n$-sets $X_n$ and maps $\varphi_n:X_n\ra X_{n+1}$. The first definitions are obvious.
\begin{defn}
a) The sequence $(X_n,\varphi_n)_{n\geq0}$ of $\mc{S}_n$-sets is \emph{consistent} if and only if the map $X_n\stackrel{\varphi_{n}}{\to}Res_{\mc{S}_{n}}^{\mc{S}_{n+1}}(X_{n+1})$ is $\mc{S}_n$-equivariant.

\ b) The sequence is \emph{injective} if $\varphi_n$ is (eventually) injective.

\ c) The sequence is $\mc{S}_*$-\emph{surjective} if $\mc{S}_{n+1}\cdot \varphi_n (X_n)=X_{n+1}$ for large $n$.
 \end{defn}

 To define "stability" we need a "stable notation" for irreducible $\mc{S}_n$-sets: these are of the form $\mc{S}_n\diagup H$, where $H$ is a subgroups of $\mc{S}_n$ defined up to conjugation.
\begin{defn}
 We say that $\mc{S}_n\diagup H$ is of the type $(\lambda_1,\lambda_2,\ldots,\lambda_t)$ $(\lambda_1\geq\lambda_2\geq\ldots\geq\lambda_t\geq1)$ if the action of $H$ on the set $\un{n}=\{1,2,\ldots,n\}$ has $t$ orbits and their cardinalities are $\lambda_1,\lambda_2,\ldots,\lambda_t$. We say that an irreducible $\mc{S}_n$-set is of \emph{stable type} $(\mu_1,\ldots,\mu_s)_n$ if this set is of type $(n-\sum\mu_i, \mu_1,\ldots,\mu_s)$, here $n-\sum \mu_i\geq\mu_1$.
\end{defn}
\begin{rems}
1) If $H$ and $K$ are conjugate in $\mc{S}_n$, their orbit types $(\lambda_1,\lambda_2,\ldots,\lambda_t)$ coincide.

\ 2) If $\mc{S}_n\diagup H$ is of the type $(\lambda_1,\lambda_2,\ldots,\lambda_t)$, then (up to conjugation) $H$ is a subgroup of $\mc{S}_{\lambda_1}\times \mc{S}_{\lambda_2}\times\ldots\times\mc{S}_{\lambda_t}$ and $pr_i(H)$ is a subgroup of $\mc{S}_{\lambda_i}$ acting transitively on the set $\un{\lambda_i}$ of cardinality $\lambda_i$.

\ 3) In general there are many irreducible $\mc{S}_n$-sets of the same type $(\lambda_1,\lambda_2,\ldots,\lambda_t)$. There is a minimal one corresponding to the largest subgroup, $\mc{S}_{\lambda_1}\times \mc{S}_{\lambda_2}\times\ldots\times\mc{S}_{\lambda_t}$. Its linearization, $L(\mc{S}_n\diagup{\mc{S}_{\lambda_1}\times \mc{S}_{\lambda_2}\times\ldots\times\mc{S}_{\lambda_t}}),$ is the Specht module $U_{(\lambda_1,\lambda_2,\ldots,\lambda_t)}$  containing the irreducible representation $V_{(\lambda_1,\lambda_2,\ldots,\lambda_t)}$ (with multiplicity one).
\end{rems}

\begin{defn}{\label{def4.4}}
 A consistent sequence of $\mc{S}_n$-sets $(X_n,\varphi_n)_{n\geq 0}$ is \emph{action stable} if for any sequence $\mu_*=(\mu_1,\ldots,\mu_t)$ there is a natural number $N_{\mu}$ such that, for any $n\geq N_{\mu}$ the following conditions are satisfied:

 a) $\varphi_n$ is injective and $\mc{S}_n$-surjective;

 b) If $Y_n\subset X_n$ and $Y_{n+1}\subset X_{n+1}$ are the disjoint union of $\mu(n)$ orbits of the stable type $(\mu_1,\ldots,\mu_t)_n$ of $X_n$ and $\mu(n+1)$ orbits of stable type $(\mu_1,\ldots,\mu_t)_{n+1}$ of $X_{n+1}$ respectively, then $\mu(n)=\mu(n+1)$.

 The sequence is \emph{uniformly action stable} if in the previous definition we can take $N$ independent of $\mu$.

 The sequence $(X_n,\varphi_n)_{n\geq 1}$ is \emph{strongly action stable} if the next condition is also satisfied:

 c) with the previous notation we have $\mc{S}_{n+1}\cdot \varphi_n(Y_n)=Y_{n+1}.$
\end{defn}

\begin{example}
The sequence $(\mc{S}_n/\mc{A}_n=\mathbb{Z}_2, Id)$ is strongly action stable. More generally, any consistent, injective and $\mc{S}_n$-surjective sequence of transitive actions whose isotopy groups acts transitively on $\{1,2,\ldots,n\}$ is strongly action stable.
\end{example}

\begin{example}
The sequence of natural $\mc{S}_n$-sets, $\un{n}=\{1,2,\ldots,n\}$ (with canonical inclusion $i_n:\un{n}\hookrightarrow\un{n+1}$) is strongly action stable.
\end{example}

\noindent The next proposition generalizes this last example.

\begin{prop}{\label{prop4.7}}
a) The sequences $(\mc{P}_k({n}))_{n\geq0}$ are strongly action stable.

b) The sequence $(\mc{P}({n}))_{n\geq0}$ is action stable.
\end{prop}

\begin{proof}
a) Like in Section 2 we take $n\geq2k$. The $\mc{S}_n$-set $\mc{P}_k({n})$ is irreducible, and the corresponding subgroup is $\mc{S}_{n-k}\times S_k$, which is of stable type $(k)_n$ with multiplicity $1$. Obviously the canonical inclusions $\mc{P}_k({n})\hookrightarrow\mc{P}_k({n+1})$ are consistent, injective and $\mc{S}_n$-surjective.

b) For $n+1\geq2k$, in $\mc{P}({n})$ there are two irreducible $\mc{S}_n$-sets of type $(n-k,k)$: $\mc{P}_k({n})$ and $\mc{P}_{n-k}({n})$, and the (stable) multiplicity of type $(k)$ is 2. The condition c) from definition 5.3 is not satisfied in this case: for the type $(k)$, the $\mc{S}_{n+1}$-span of the image $\varphi_n(Y_n)=\varphi(\mc{P}_k({n})\sqcup \mc{P}_{n-k}({n}))$ is $\mc{P}_k({n+1})\sqcup \mc{P}_{n-k}({n+1})$, different from $Y_{n+1}=\mc{P}_k({n+1})\sqcup\mc{P}_{n+1-k}({n+1})$.
\end{proof}

\section{Canonical polynomial basis}{\label{sec5}}
Now we translate the power set representations into a quotient representation of the polynomial algebra $\mathbb{Q}[x_1,\ldots,x_n]$; we compute canonical basis for the irreducible components in this isomorphic algebraic model. The set of squares $sq(n)=\{x_1^2,\ldots,x_n^2\}$ is $\mc{S}_n$-invariant, hence the ideal generated by $sq(n)$ is $\mc{S}_n$-invariant and we obtain a quotient representation of $\mc{S}_n$ on the space of "square free" polynomials $$\mc{S}f(n)=\mathbb{Q}[x_1,\ldots,x_n]\diagup \langle sq(n)\rangle\cong \mathbb{Q}\langle\bold{\underline{x}}_A\mid A\in\mc{P}(n)\rangle$$ where $\mathbf{\underline{x}}_A=x_{a_1}x_{a_2}\ldots x_{a_k}$ is the square free monomial corresponding to the subset $A=\{a_1,\ldots,a_k\}\in \mc{P}_k(n)$.

\begin{lem}
The power set $L\mc{P}(n)$ and the space of the square free polynomials $\mc{S}f(n)$ are isomorphic $\mc{S}_n$-modules.
\end{lem}

\begin{proof}
The power set $\mc{P}(n)$ and the canonical basis $\{\mathbf{\un{x}}_A\mid A\in\mc{P}(n)\}$ are isomorphic as $\mc{S}_n$-sets.
\end{proof}

\noindent In the new setting we have a $\mc{S}_n$-decomposition by grading $\mc{S}f(n)=\mathop{\bigoplus}\limits_{k=0}^n \mc{S}f_k(n)$ and $$F_{\ast}\mc{S}f_k(n): 0<F_0\mc{S}f_k(n)<F_1\mc{S}f_k(n)<\ldots<F_k\mc{S}f_k(n)=\mc{S}f_k(n),$$ the  $\mc{S}_n$-filtration and also the irreducible components:
\begin{cor}
For any $i,k,n$ satisfying $0\leq i\leq k\leq \frac{n}{2}$ we have
\begin{center}$\begin{array}{rll}
  \mc{S}f_k(n)  & \cong & \mc{S}f_{n-k}(n), \\
  F_i\mc{S}f_k(n) & = & V_{(n)}\oplus V_{(n-1,1)}\oplus\ldots \oplus V_{(n-i,i)}, \\
  \mc{S}f_k(n) & = & V_{(n)}\oplus V_{(n-1,1)}\oplus\ldots \oplus V_{(n-k,k)}, \\
  \mc{S}f(n) & = & (n+1)V_{(n)}\oplus \ldots \oplus(n-2k+1) V_{(n-k,k)}\oplus \ldots\oplus r V_{(\lceil\frac{n}{2}\rceil,\lfloor\frac{n}{2}\rfloor)},
   %\oplus (n-1)V(n-1,1)\oplus\ldots  \\
   %&  & \ldots \oplus(n-2k+1) V(n-k,k)\oplus \ldots\oplus r V(\lceil\frac{n}{2}\rceil,\lfloor\frac{n}{2}\rfloor),
\end{array}$\\
\end{center}
where $r=\lceil\frac{n}{2}\rceil-\lfloor\frac{n}{2}\rfloor+1.$
\end{cor}

\begin{cor}
a) The sequence $(\mc{S}f(n))_{n\geq0}$ given by the canonical inclusion \\ $\mc{S}f(n)\hookrightarrow \mc{S}f(n+1)$ is consistent, representation stable and monotonic.

b) The sequences $(F_i\mc{S}f_k(n))_{n\geq0}$ and $(\mc{S}f_k(n))_{n\geq0}$ are consistent, uniformly representation stable and monotonic.
\end{cor}

\noindent We will use the following notations for some polynomials in $\mc{S}f(n)$ (the first one is the symmetric fundamental polynomial in $n$ variables, the third corresponds to the previous $\sigma_k^n(A)$).\\
\begin{center}
$\begin{array}{rll}
  \sigma_k^n     & = & \mathop{\sum}\limits_{1\leq i_1\ldots\leq i_k\leq n}x_{i_1}x_{i_2}\ldots x_{i_k}     \\
  \sigma_k^B     & = & \mathop{\sum}\limits_{C\in \mc{P}_k(B)}\mathbf{\un{x}}_{C}   =  \mathop{\sum}\limits_{b_i\in B}x_{b_1}x_{b_2}\ldots x_{b_k}  \\
   \sigma_k^n (A)&  = & \mathbf{\un{x}}_A\sigma_{k-|A|}^{A'}   =  x_{a_1}\ldots x_{a_i}(\mathop{\sum}\limits_{b_j\not\in A}x_{b_{i+1}}\ldots x_{b_k})
   \end{array}$
\end{center}
(in the last formulae, $\mbox{card}(B)\geq k\geq \mbox{card}(A)$, $A=\{a_1,\ldots,a_i\}$, and $A'=\un{n}\setminus A$) and also
$$\begin{array}{c}
   \d_{hj}=x_h-x_j$ $(h<j) \\
   \d_{H_* J_*}=\d_{h_1 j_1}\d_{h_2 j_2}\ldots\d_{h_s j_s}
 \end{array}$$

 \noindent where $H_*=(h_1,h_2,\ldots,h_s)$, $J_*=(j_1,j_2,\ldots,j_s)$, $h_{\a}<j_{\a}$ and $H_*\cup J_*$ contains $2s$ elements. Using these notations, $\{\sigma_k^n(A)\mid \mathrm{card}(A)=i\}$ is a basis of $F_k\mc{S}f_k(n)$. Now we will describe bases for the irreducible $\mc{S}_n$-submodules $V_{(n-j,j)}$ contained in $\mc{S}f_k(n)$. For this, the next remark is fundamental:

\begin{rem}
The space $\mc{S}f(n)$ (like $L\mc{P}(n)$) has a canonical inner product: $$\lan \mathbf{\un{x}}_A,\mathbf{\un{x}}_B\ran=\delta_{A,B}$$ and the natural representation of $\mc{S}_n$ is an orthogonal representation. The homogenous components $\mc{S}f_k(n)$ are orthogonal and the irreducible pieces $V_{(n-j,j)}$ are pairwise orthogonal (even those in the same isotypical component: they lie in different homogenous components).
\end{rem}

\noindent Our method is to find vectors in $F_{i-1}\mc{S}f_k(n)^\bot$, the orthogonal complement of $F_{i-1}\mc{S}f_k(n)$ in $F_i\mc{S}f_k(n)$, because this complement corresponds to the irreducible component $V_{(n-i,i)}$ of $\mc{S}f_k(n)$, next to describe an independent subset of these vectors, and finally the computation of cardinality and dimension will give the basis.

\begin{lem}
The following vectors from $\mc{S}f_k(n)$ are orthogonal to $F_{k-1}\mc{S}f_k(n)$: $$\{\d_{H_*J_*}\mid H_*=(h_1,h_2,\ldots,h_k), J_*=(j_1,\ldots,j_k),h_{\a}\leq j_{\a},\mathrm{card}(H\cup J)=2k\}.$$
\end{lem}

\begin{proof}
Obviously $\d_{H_*J_*}\in \mc{S}f_k(n)$. The canonical basis of $F_{k-1}\mc{S}f_k(n)$ is given by $\{\sigma_k^n(A)\mid \mathrm{card}(A)=k-1\}$. Computing $\langle\mathbf{\un{x}}_A\cdot\sigma_1^{A'},\d_{h_1j_1}\d_{h_2j_2}\ldots\d_{h_kj_k}\rangle$ we obtain:\\
  \indent a) if $A\nsubseteq H\cup J$, there is no match between the monomials of these two polynomials;\\
%\noindent
 \indent b) if there is an index $s\epsilon\{1,\ldots,k\}$ such that $\{h_s,j_s\}\subset A$, there is no match between the monomials of the two polynomials;\\
% \noindent
 \indent c) if $A$ contains $\a$ elements $H_{\a}\subset H_*$ and $\b$ elements $J_{\b}\subset J_*$ ($\a +\b =k-1$ and the indices of these elements are disjoint), there are precisely two common monomials of the given polynomials: $\mathbf{\un{x}_A}x_{h_s}$ and $\mathbf{\un{x}_A}x_{j_s}$ where the index $s$ is the unique index from 1 to $k$ which does not appear as an index in $H_{\a}\cup J_{\b}$; the two monomials have coefficients (1,1) in the first polynomial and $(\pm1,\mp1)$ in the second one.

 Therefore in all three cases the dot product is zero.
\end{proof}

\begin{lem}
The following vectors from $F_i\mc{S}f_k(n)$ are orthogonal to $F_{i-1}\mc{S}f_k(n)$: $$\{\d_{H_*J_*}\sigma^L_{k-i}\mid H_*=(h_1,\ldots,h_i), J_*=(j_1,\ldots,j_i), H\sqcup J\sqcup L=\un{n}\}$$
\end{lem}

\begin{proof}
Translating Lemma 2.4 into polynomial notation we obtain the isometry of pairs of Euclidean spaces (it transforms an orthogonal basis into an orthogonal basis):
$$\Psi:(\mc{S}f_i(n),F_{i-1}\mc{S}f_i(n))\stackrel{\cong}{\to}(F_i\mc{S}f_k(n),F_{i-1}\mc{S}f_k(n)) ,$$ $$\mathbf{\underline{x}}_A\mapsto \mathbf{\underline{x}}_A \sigma_{k-i}^{A'}.$$ A direct computation shows that $$\Psi(\d_{H_*J_*})=\Psi(\mathop{\sum}\limits_{\alpha \sqcup \beta=\un{i}}(-1)^{|\beta|}\mathbf{\underline{x}}_{H_\alpha}\mathbf{\underline{x}}_{J_\beta})=\mathop{\sum}\limits_{\alpha \sqcup \beta=\un{i}}(-1)^{|\beta|}\mathbf{\underline{x}}_{H_\alpha}\mathbf{\underline{x}}_{J_\beta}\sigma_{k-i}^{(H_\alpha\sqcup J_\beta)'} .$$ Using the decomposition formula $\sigma_p^{X\sqcup Y}=\sigma_p^X+\sigma_p^Y+ \mathop{\sum}\limits_{q+q'=p} \sigma_q^X\sigma_{q'}^Y$, the symmetric sum $\sigma_{k-i}^{(H_\alpha\sqcup J_\beta)'}$   splits into $$\sigma_{k-i}^{(H_*\sqcup J_*)'}+\sigma_{k-i}^{H_{\a'}\sqcup J_{\b'}}+ \mathop{\sum}\limits_{q+q'=k-i} \sigma_q^{H_{\a'}\sqcup J_{\b'}}\cdot \sigma_{q'}^{(H_*\sqcup J_{\ast})'},$$ where $H_{\alpha'}=H_*\setminus H_{\alpha}\,\,and\,\, J_{\b'}=J_{\ast}\setminus J_{\b}.$ The first sum in this splitting gives the desired result: $$\mathop{\sum}\limits_{\alpha \sqcup \beta=\un{i}}(-1)^{|\beta|}\mathbf{\underline{x}}_{H_\alpha}\mathbf{\underline{x}}_{J_\beta}\sigma_{k-i}^{(H_*\sqcup J_*)'}=\d_{H_*J_*}\sigma_{k-i}^L.$$  We will show that the second and the third term have no contribution to $\Psi(\d_{H_*J_*})$ and this will finish the proof: through the isometry $\Psi$ the subspace $F_{i-1}\mc{S}f_i(n)$ is sent to $F_{i-1}\mc{S}f_k(n)$, hence $\Psi(F_{i-1}\mc{S}f_i(n)^{\perp})=F_{i-1}\mc{S}f_k(n)^{\perp}.$ To show that the second sum $$\mathop{\sum}\limits_{\alpha \sqcup \beta=\un{i}}(-1)^{|\beta|}\mathbf{\underline{x}}_{H_\alpha}\mathbf{\underline{x}}_{J_\beta}\sigma_{k-i}^{H_{\a'}\sqcup J_{\b'}}$$ and the third sum\\ $$\mathop{\sum}\limits_{\alpha \sqcup
\beta=\un{i}}\big[(-1)^{|\beta|}\mathbf{\underline{x}}_{H_\alpha}\mathbf{\underline{x}}_{J_\beta}\mathop{\sum}\limits_{q+q'=k-i} \sigma_{q'}^{H_{\a'}\sqcup J_{\b'}}\cdot \sigma_{q'}^{(H_*\sqcup J_{\ast})'}\big]=$$ $$=\mathop{\sum}\limits_{q+q'=k-i}\big[\mathop{\sum}\limits_{\alpha \sqcup \beta=\un{i}}(-1)^{|\beta|}\mathbf{\underline{x}}_{H_\alpha}\mathbf{\underline{x}}_{J_\beta} \sigma_{q}^{H_{\a'}\sqcup J_{\b'}}\cdot \sigma_{q'}^{(H_*\sqcup J_{\ast})'}\big]$$ are zero, it is enough to prove that for any $q$ in the interval $[1,k-i]$ the next sum is zero  $$S=\mathop{\sum}\limits_{\alpha \sqcup \beta=\un{i}}(-1)^{|\beta|}\mathbf{\underline{x}}_{H_\alpha}\mathbf{\underline{x}}_{J_\beta} \sigma_{q}^{H_{\a'}\sqcup J_{\b'}}.$$ This sum contains monomials from two disjoint set of variables, $\{x_{h_1},\ldots,x_{h_i}\}$ and $\{x_{j_1},\dots,x_{j_i}\}$ ($H_*=(h_1,\ldots,h_i), J_*=(j_1,\ldots,j_i)$). Therefore in such a monomial $\mathbf{m}=\mathbf{\underline{x}}_{H_\alpha}\mathbf{\underline{x}}_{J_\beta} \mathbf{\underline{x}}_{M}$ ($M$ is a $q$-subset of $H_{\a'}\sqcup J_{\b'}$) there are indices $p$ such that $x_{h_p}$ and $x_{j_p}$ are both contained in $\mathbf{m}$; on the other hand, precisely one of them is in the "first part" and the other is in the "second part": either $h_p\in H_{\a},\,\,j_p\in M$ or $h_p\in M,\,\, j_p\in J_{\b}.$ We define an involution on the set of monomials $\mathbf{m}$ in $S$, choosing the maximal common index $p$ and changing the places of $x_{h_p}$ and $x_{j_p}$ ($h_p\in H_{\a}$):$$\mathbf{m}=\mathbf{\underline{x}}_{H_\alpha}\mathbf{\underline{x}}_{J_\beta}\mathbf{\underline{x}}_M \leftrightarrow \mathbf{m'}= \mathbf{\underline{x}}_{H_\alpha \setminus \{h_p\}}\mathbf{\underline{x}}_{J_\beta \sqcup \{j_p\}}\mathbf{\underline{x}}_{M\sqcup \{h_p\}\setminus \{j_p\}} .$$ In $S$, these two monomials have coefficients $(-1)^{|\b|}\mathbf{m}$ and $(-1)^{|\b|+1}\mathbf{m'}$, hence the total sum is zero.
\end{proof}

\begin{lem}\label{lemma 5.7}
The set $$\{\d_{H_*J_*}\in \mc{S}f_k(n)\mid H_*=(h_1,\ldots,h_k), J_*=(j_1,\ldots,j_k), h_{\a}< j_{\a}\,\,, \emph{card}(H_*\cup J_*)=2k\}$$ contains a linearly independent set of cardinality ${n\choose k} - {n\choose {k-1}}$.
\end{lem}

\begin{proof}
By induction on $n$, we start with $n=0$, $n=1$, $n=2$:
\begin{center}\begin{tabular}{rll}

  % after \\: \hline or \cline{col1-col2} \cline{col3-col4} ...
  $n=0$: & $\mc{S}f_0(0)=1$; &  \\
  $n=1$: & $F_0\mc{S}f_0(1)=\lan x_1\ran$; &  \\
  $n=2$: & $F_0\mc{S}f_1(2)=\lan x_1+x_2\ran$, & $F_0\mc{S}f_1(2)^{\perp}=\lan x_1-x_2\ran$. \\
  \end{tabular}\end{center}
\noindent Suppose we have a linearly independent set of polynomials $\d_k^{n-1}=\{ \d_{H_*J_*}\in \mc{S}f_k(n-1)\}$ of cardinality $\delta_k^{n-1}={{n-1}\choose {k}}-{{n-1}\choose {k-1}}$ and a second set of linearly independent polynomials $\d_{k-1}^{n-1}=\{\d_{H_*J_*}\in \mc{S}f_{k-1}(n-1)\}$ of cardinality $\delta_{k-1}^{n-1}={{n-1}\choose {k-1}} -{{n-1}\choose{k-2}}$. Then we can define a subset in $\mc{S}f_k(n)$ taking $\d_k^{n-1}$ and all polynomials $\d_{r,n}\cdot \d_{L_*M_*}$ with $\d_{L_*M_*}$ in $\d_{k-1}^{n-1}$, where the index $r$ is the smallest element in the complement of $L_*\sqcup M_* \sqcup \{n\}$. These polynomials are linearly independent:
$\sum c_{r,L_*M_*} \d_{r,n}\cdot \d_{L_*M_*}+ \sum c_{H_*L_*} \d_{H_*L_*}=0$ implies $c_{r,L_*,M_*}=0$ (look at the coefficient of $x_n$) and next $c_{H_*J_*}=0$. Their total number is \begin{center} $\delta_k^{n-1} +\delta_{k-1}^{n-1}={{n-1}\choose k} + {{n-1}\choose k-1}-{n-1\choose k-1}-{n-1\choose k-2}={n\choose k}-{n\choose k-1}.$\end{center}
\end{proof}

\begin{prop}{\label{prop 5.8}} There is a basis of the irreducible component $V_{(n-i,i)}$ of $\mc{S}f_k(n)$ given by polynomials of the form
$$\{\d_{H_*J_*} \sigma_{k-i}^L\mid H_*=(h_1,\ldots,h_i), J_*=(j_1,\ldots,j_i), H\sqcup J\sqcup L=\un{n}\}.$$
\end{prop}

\begin{proof} The irreducible component $V_{(n-i,i)}$ of $\mc{S}f_k(n)$ is the orthogonal complement of $F_{i-1}\mc{S}f_k(n)$ in $F_i\mc{S}f_k(n)$; its dimension is ${n\choose i}-{n\choose i-1}$. From Lemma \ref{lemma 5.7} there is a set of linearly independent polynomials $\{\d_{H_*J_*}\in \mc{S}f_{i}(n)\}$ of cardinality ${n\choose i}-{n\choose i-1}$. The image of this set through the isometry $\Psi: \mc{S}f_i(n)\rightarrow F_i\mc{S}f_k(n)$ gives the required basis.
\end{proof}

\begin{example}
Using the algorithm described in Lemma \ref{lemma 5.7} and Proposition \ref{prop 5.8}, we find the basis of the component $V(5,2)$ of $\mc{S}f_3(7)$:\\
$$\begin{array}{ll}
  (x_4-x_1)(x_2-x_3)(x_5+x_6+x_7), & (x_6-x_2)(x_1-x_4)(x_3+x_5+x_7), \\
  (x_4-x_2)(x_1-x_3)(x_5+x_6+x_7) ,& (x_6-x_2)(x_1-x_5)(x_3+x_4+x_7), \\
  (x_5-x_3)(x_1-x_2)(x_4+x_6+x_7) ,& (x_7-x_3)(x_1-x_2)(x_4+x_5+x_6), \\
  (x_5-x_2)(x_1-x_3)(x_4+x_6+x_7), & (x_7-x_2)(x_1-x_3)(x_4+x_5+x_6) ,\\
   (x_5-x_2)(x_1-x_4)(x_3+x_6+x_7),& (x_7-x_2)(x_1-x_4)(x_3+x_5+x_6), \\
  (x_6-x_2)(x_1-x_2)(x_4+x_6+x_7),& (x_7-x_2)(x_1-x_5)(x_3+x_4+x_6), \\
  (x_6-x_2)(x_1-x_3)(x_3+x_5+x_7) ,& (x_7-x_2)(x_1-x_6)(x_3+x_4+x_5).
\end{array}$$
\end{example}

\section{An application to the Arnold Algebra}
V.I. Arnold \cite{A} computed the cohomology algebra of the pure braid group $P_n$, describing the first non trivial cohomology algebra of a complex hyperplane arrangement, later generalized by Orlik-Solomon \cite{OS} to arbitrary hyperplane arrangements. We denote this algebra by $\mc{A}_n$.
\begin{defn}
(Arnold) The Arnold algebra $\mc{A}_n$ is the graded commutative algebra (over $\mathbb{Q}$) generated in degree one by $n\choose2$ generators $\{w_{ij}\}$ having the defining relations of degree two (the Yang Baxter or the infinitesimal braid relations) $YB_{ijk}$: $$\mc{A}_n=\lan w_{ij},\,\,1\leq i<j\leq n \mid YB_{ijk}: w_{ij}w_{ik}-w_{ij}w_{jk}+w_{ik}w_{jk},\,\,1\leq i<j<k\leq n\ran.$$
\end{defn}
\noindent With the convention $w_{ij}=w_{ji}$ $(i\neq j)$, we define the natural action of the symmetric group $\mc{S}_n$ on the exterior algebra $\Lambda^*(w_{ij})$ by $\pi\cdot w_{ij}=w_{\pi(i)\pi(j)}$. The set of infinitesimal braid relations $\{YB_{ijk}\}$ is invariant (up to a sign) so we have a natural action of $\mc{S}_n$ onto the Arnold algebra $\mc{A}_n$. Church and Farb \cite{CF} proved the representation stability of $\mc{A}_n$ (see also \cite{H}). We will use some of our previous results to describe the irreducible $\mc{S}_n$ submodules of $\mc{A}^1(n)$ and $\mc{A}^2(n)$.

First, some notations:
\begin{center}$\begin{array}{lll}
  \Omega^n & = & w_{12}+w_{13}+\ldots+w_{n-1,n}, \\
  \Omega^n_{ij} & = & \mathop\sum\limits_{k\neq i,j}(w_{ik}-w_{jk}), \\
  \Omega_{ijkl} & = & w_{il}-w_{ik}+w_{jk}-w_{jl}.
\end{array}$ \end{center}

\begin{prop}The degree one component of the Arnold algebra decomposes as $$\mc{A}^1(n)= V(0)_n\oplus V(1)_n\oplus V(2)_n$$ for $n\geq4$. In the unstable cases the decompositions are $$\mc{A}^1(2)=V_{(2)},\,\,\mc{A}^1(3)=V_{(3)}\oplus V_{(2,1)}.$$
The following list gives bases of these three components: \\
\begin{center}$\begin{array}{lll}
  \mc{B}(n) & = & \{\Omega^n\} \\
  \mc{B}(n-1,1) & = & \{\Omega^n_{12},\Omega^n_{13},\ldots,\Omega^n_{1n}\} \\
  \mc{B}(n-2,2) & = & \{\Omega_{1234},\Omega_{1324}, \\
   &  & \Omega_{1235},\Omega_{1325},\Omega_{1425}, \\
   &  & \Omega_{1236},\Omega_{1326},\Omega_{1426},\Omega_{1526}, \\
   &  & \ldots \ldots\ldots\ldots \\
   &  & \Omega_{123n},\Omega_{132n},\Omega_{142n},\ldots,\Omega_{1,n-1,2,n}\}.
\end{array}$\\
\end{center}
%$\Omega^n= w_{12}+w_{13}+\ldots+w_{1n}$, $\Omega^n_{ij}=\mathop\sum\limits_{k\neq i,j}(w_{ik}-w_{jk})$, and $\Omega_{ijkl}=w_{il}-w_{ik}+w_{jk}-w_{jl}$.
\end{prop}
\begin{proof}The first part is a consequence of the isomorphism of $\mc{S}_n$-modules $\mc{A}^1(n)\cong L\mc{P}_2(n)$, $w_{ij}\mapsto \{i,j\}$, and of Proposition \ref{prop 2.6}. The second part is a consequence of the inductive method to construct bases of the different pieces of $\mc{S}f_2(n)\cong L\mc{P}_2(n)\cong \mc{A}^1(n)$: for instance, the polynomial $$\d_{12}\sigma^{(12)'}_n= (x_1-x_2)(x_3+x_4+\ldots+x_n)$$ correspond to the linear combination of sets $$(\{1,3\}+\{1,4\}+\ldots+\{1,n\})-(\{2,3\}+\{2,4\}+\ldots+\{2,n\})$$ and this corresponds to $\Omega^n_{12}$. Similarly, the polynomial $\d_{ij}\d_{lk}=(x_i-x_j)(x_l-x_k)$ corresponds to $\Omega_{ijkl}$.
\end{proof}
\noindent
The vector space $L\mc{P}_3(n)$ is isomorphic to $I^2(n)$, the degree two component of the ideal of infinitesimal braid relations:
$$\{i,j,k\}\leftrightarrow YB_{ijk}:w_{ij}w_{ik}-w_{ij}w_{jk}+w_{ik}w_{jk},$$ but they are not isomorphic as $\mc{S}_n$-modules: the symmetric group action on $I^2(n)$ involves signs. For instance $(12)\cdot YB_{123}=w_{12}w_{23}-w_{12}w_{13}+w_{23}w_{13}=-YB_{123}.$

\begin{prop}\label{prop 6.3}
For $n\geq4$ the degree two component of the ideal of relations decomposes as $$I^2(n)=V(1,1)_n\oplus V(1,1,1)_n.$$
For $n=2$ we have $I^2(2)=0$ and for $n=2$ and $n=3$ we have $I^2(n)=0$ and $I^2(3)=V_{(1,1,1)}.$
\end{prop}
\begin{proof}
The characters of the irreducible modules $V_{(n-2,1,1)}$ and $V_{(n-3,1,1,1)}$ can be computed using the Frobenius formula and are given in the character table in the proof of the next lemma. We obtain the character of $I^2(n)$ by direct computation. The symmetric group acts on the canonical basis $\{YB_{ijk}\}$ of $I^2(n)$ by permuting the elements of this basis and adding a $\pm$ sign. The relation $YB_{ijk}$ is invariant (up to sign) by a permutation $\pi$ if and only if $\{i,j,k\}$ is a union of cycles of $\pi$. If the permutation $\pi$ has type $(i_1,i_2,\ldots,i_n)$ ($i_q$ is the number of cycles of length $q$), then $\pi$ leaves invariant the elements $YB_{ijk}$ corresponding to three fixed points $i,j,k$ (the number of relations of this first type is $i_1\choose 3$) and also elements $YB_{pqr}$ corresponding to a three cycle $(pqr)$ (and the number of relations of this second type is $i_3$). In the last case, if $\{i\}$ is a fixed point of $\pi$ and $(uv)$ is a two-cycle, we have $\pi\cdot Y_{iuv}=-Y_{iuv}$ (and the total number of such relations is $i_1i_2$). Therefore the value of the character on $\pi$ is $$\chi_{I^2(n)}(i_1,i_2,\ldots,i_n)={i_1\choose3}+i_3-i_1i_2,$$ and it is equal to $\chi_{V(n-2,1,1)}+\chi_{V(n-3,1,1,1)}$.
\end{proof}

\begin{lem}\label{lemma6.4}
For $n\geq7$ the degree two component of the exterior algebra $\Lambda^2(n)=\Lambda^*(w_{ij})_{1\leq i<j\leq n}$ decomposes as $$\Lambda^2(n)=2V(1)_n\oplus 2V(2)_n\oplus3V(1,1)_n\oplus V(3)_n\oplus 2V(2,1)_n\oplus V(1,1,1)_n\oplus V(3,1)_n.$$ The unstable cases have the decompositions:\\
\begin{center}$\begin{array}{rcl}
   \Lambda^2(2)&=& 0,\\
   \Lambda^2(3) & = & V_{(2,1)}\oplus V_{(1,1,1)}, \\
   \Lambda^2(4) & = &2 V_{(3,1)}\oplus V_{(2,2)}\oplus2 V_{(2,1,1)} \oplus V_{(1,1,1,1)}, \\
   \Lambda^2(5) & = & 2V_{(4,1)} \oplus 2V_{(3,2)}\oplus 3V_{(3,1,1)}\oplus V_{(2,2,1)}\oplus V_{(2,1,1,1)},\\
   \Lambda^2(6) & = & 2V_{(5,1)} \oplus 2V_{(4,2)}\oplus 3V_{(4,1,1)}\oplus V_{(3,3)}\oplus 2V_{(3,2,1)}\oplus V_{(3,1,1,1)}.
 \end{array}
$\end{center}
\end{lem}
\begin{proof}
These decompositions are obtained from the expansion \\
\begin{center}$\begin{array}{rl}
  \Lambda^2(\mc{A}^1) & =\Lambda^2(V(0)_n\oplus V(1)_n\oplus V(2)_n) \\
   &  =\Lambda^2V(1)_n\oplus \Lambda^2V(2)_n\oplus V(1)_n\oplus V(2)_n\oplus (V(1)_n\otimes V(2)_n),
\end{array}$\end{center}

%$$\Lambda^2(\mc{A}^1)=\Lambda^2(V(0)_n\oplus V(1)_n\oplus V(2)_n)=\Lambda^2V(1)_n\oplus \Lambda^2V(2)_n\oplus V(1)_n\oplus V(2)_n\oplus (V(1)_n\otimes V(2)_n),$$\\
 where\\
 \begin{center}$\begin{array}{cll}
   V(1)_n\otimes V(2)_n & = & V(1)_n\oplus V(2)_n\oplus V(1,1)_n\oplus V(3)_n\oplus V(2,1)_n, \\
   \Lambda^2V(1)_n & = & V(1,1)_n, \\
   \Lambda^2V(2)_n & = & V(1,1)_n\oplus V(2,1)_n\oplus V(1,1,1)_n\oplus V(3,1)_n.
 \end{array}$\\ \end{center}
 The decompositions of the tensor product is from \cite{M} (and can be checked using Littlewood-Richardson rule or using the characters from the next table). For the degree two exterior algebra one can use the next character table ($(i_1,i_2,\ldots,i_n)$ stands for the conjugacy class with $i_q$ cycles of length $q$):\\
\begin{center} \begin{tabular}{|c|c|c|c|}
   \hline
   % after \\: \hline or \cline{col1-col2} \cline{col3-col4} ...
    & $\chi_{V}(i_1,\ldots,i_n)$ & $\psi_{V}^{2}(i_1,\ldots,i_n)$ & $\chi_{\Lambda^2V}(i_1,\ldots,i_n)$ \\ \hline
   $V(1)_n$ & $i_1-1$ & $i_1+2i_2-1$ & ${i_2-1\choose2}-i_2$ \\ \hline
   $V(1,1)_n$ & ${i_1-1\choose 2}-i_2$ &  &  \\ \hline
   $V(2)_n$ & $\frac{i_1(i_1-3)}{2}+i_2 $& $\frac{(i_1+2i_2)(i_1+2i_2-3)}{2}+$ & $\frac{i_1(i_1-3)(i_1^2-3i_1-2)}{8}+$\\
            &                            &                                $+2i_4$        & $+\frac{(i_1^2-5i_1+3)i_2-i_2^2}{2}-i_4$ \\ \hline
   $V(3)_n$ & $\frac{i_1(i_1-1)(i_1-5)}{6}+$  &                                         &  \\
            &           $+i_2(i_1-1)+i_3$                                     &                                        &\\ \hline
   $V(2,1)_n$ & $\frac{i_1(i_1-2)(i_1-4)}{3}-i_3$ &  &  \\ \hline
   $V(1,1,1)_n$ & ${i_1-1\choose 3}+i_2(1-i_1)+i_3 $&  &  \\ \hline
   $V(3,1)_n$ & $\frac{i_1(i_1-1)(i_1-3)(i_1-6)}{8}+ $&  &  \\
                & $+i_2{i_1-1\choose 2}-{i_2\choose 2}-i_4 $                                   &   &  \\
   \hline
 \end{tabular}\\
 \end{center}
\noindent The entries in the second column are computed using the Frobenius formula, in the third column we used $$(i_1,i_2,i_3,i_4,\ldots)^2=(i_1+2i_2, 2i_4,\ldots),$$ and in the last column we used the determinant formula
\begin{center}$\chi_{\Lambda^2(V)}=\frac{1}{2}$\begin{tabular}{|c c|}
% after \\: \hline or \cline{col1-col2} \cline{col3-col4} ...
$\psi^1_V$ & 1 \\
$\psi^2_V$ & $\psi^1_V $\\
\end{tabular}
\end{center}

 \noindent(see \cite{K}).
 \end{proof}

From Proposition \ref{prop 6.3} and Lemma \ref{lemma6.4} we obtain:

\begin{prop}
For $n\geq 7$ the degree 2 component of the Arnold algebra decomposes as $$\mc{A}^2(n)=2V(1)_n\oplus 2V(2)_n\oplus 2V(1,1)_n\oplus V(3)_n\oplus 2V(2,1)_n\oplus V(3,1)_n.$$ In the unstable cases we have \begin{center}$\begin{array}{rcl}
                                              \mc{A}^2 (2)& = & 0, \\
                                              \mc{A}^2 (3) & = & V_{(2,1)}, \\
                                              \mc{A}^2 (4) & = & 2V_{(3,1)}\oplus V_{(2,2)}\oplus V_{(2,1,1)}, \\
                                              \mc{A}^2 (5) & = & 2V_{(4,1)}\oplus 2V_{(3,2)}\oplus 2V_{(3,1,1)}\oplus V_{(2,2,1)},\\
                                              \mc{A}^2 (6) & = & 2V_{(5,1)}\oplus 2V_{(4,2)}\oplus 2V_{(4,1,1)}\oplus V_{(3,3)}\oplus 2V_{(3,2,1)}.
                                            \end{array}$\end{center}

\end{prop}
These decompositions coincide with the formulae from \cite{CF}, with a single difference in $\mc{A}^2 (5)$.


\begin{thebibliography}{10}

%\bibitem[AAB]{AAB} E. Artin, \emph{Theory of braids}, Ann. of Math. (2)
%\textbf{48} (1947),  pp. 101-126.
\bibitem[A]{A} V.I. Arnold, \emph{The cohomology ring of dyed braids}, Mathematical Notes 5, no. 2 (1969), 138-140.

\bibitem[C]{C} T. Church, \emph{Homological stability for configuration spaces of manifolds}, arXiv:1103.2441v1 [math.AT] (2011).

\bibitem[CF]{CF} T. Church, B. Farb, \emph{Representation theory and homological stability}, arXiv:1008.1368v1 [math.RT] (2010).

\bibitem[FH]{FH} W. Fulton, J. Harris,  \emph{A First Course in Representation Theory}, Graduate Texts in Mathematics 129, Springer-Verlarg Berlin, 2004.

\bibitem[H]{H} D. Hemmer, \emph{Stable decompositions for some symmetric group characters arising in braid group cohomology}, Journal of Combinatorial Theory, Series A (2011), pp. 1136-1139.


\bibitem[J]{J} G.D. James, \emph{The Representation Theory of the Symmetric Groups}, Lecture Notes in Mathematics 682, Springer-Verlarg Berlin, 1978.
\bibitem[K]{K} D. Knutson, \emph{$\lambda$-Rings and the Representation Theory of the Symmetric Group}, Lecture Notes in Mathematics 308, Springer-Verlarg Berlin, 1973.
\bibitem[M]{M} F.D. Murnaghan, \emph{The analysis of the Kronecker product of irreducible representations of the symmetric group}, American Journal of Mathematics , Vol. 60, No. 3 (1938), pp.761-784.
\bibitem[OS]{OS} P. Orlik and L. Solomon, \emph{Combinatorics and topology of complements of hyperplanes}, Invent. Math. 56 (1980), no. 2, 167-189.
%%%%\bibitem[6]{G} F.A. Garside, \emph{The braid groups and other groups},Quart. J. of Math. Oxford, $2^{e}$ ser. \textbf{20} (1969), 235--254.
%V. I. Arnold. The cohomology ring of dyed braids. Mat. Zametki 5 (1969),

\end{thebibliography}
\end{document}